\documentclass[11pt,letterpaper]{amsart}

\usepackage{fullpage}
\usepackage{enumitem}
\usepackage{amsmath,amsthm,amssymb}

\newtheorem{theorem}{Theorem}[section]
\newtheorem{proposition}[theorem]{Proposition}
\newtheorem{lemma}[theorem]{Lemma}

\numberwithin{equation}{section}

\allowdisplaybreaks[1]

\renewcommand{\Re}{{\rm Re}}
\renewcommand{\Im}{{\rm Im}}
\DeclareMathOperator{\E}{E}
\newcommand{\C}{\mathbb{C}}
\newcommand{\R}{\mathbb{R}}
\newcommand{\e}{\epsilon}
\newcommand{\w}{\omega}
\newcommand{\z}{\zeta}
\newcommand{\norm}[1]{\lVert#1\rVert}
\newcommand{\abs}[1]{\lvert#1\rvert}
\newcommand{\bigabs}[1]{\big\lvert#1\big\rvert}
\newcommand{\Bigabs}[1]{\Big\lvert#1\Big\rvert}
\newcommand{\biggabs}[1]{\bigg\lvert#1\bigg\rvert}
\newcommand{\Biggabs}[1]{\Bigg\lvert#1\Bigg\rvert}
\newcommand{\leg}[2]{({#1}\!\mid\!{#2})}
\newcommand{\bigleg}[2]{\big({#1}\,\big|\,{#2}\big)}
\newcommand{\V}{\mathcal{V}}
\newcommand{\U}{\mathcal{U}}
\newcommand{\stack}[2]{\genfrac{}{}{0pt}{}{#1}{#2}}

\begin{document}

\title{The $L_4$ norm of Littlewood polynomials\\ derived from the Jacobi symbol}

\author{Jonathan Jedwab \and Kai-Uwe Schmidt}

\date{7 September 2010 (revised 4 August 2011 and 29 June 2012)}



\thanks{J.~Jedwab is with Department of Mathematics, Simon Fraser University, 8888 University Drive, Burnaby BC V5A 1S6, Canada. Email: {\tt jed@sfu.ca}.}

\thanks{K.-U.~Schmidt was with Department of Mathematics, Simon Fraser University and is now with Faculty of Mathematics, Otto-von-Guericke University, Universit\"atsplatz~2, 39106 Magdeburg, Germany. Email: {\tt kaiuwe.schmidt@ovgu.de}}

\thanks{J.~Jedwab is supported by NSERC of Canada.}

\thanks{K.-U. Schmidt is supported by German Research Foundation.}

\begin{abstract}
Littlewood raised the question of how slowly the $L_4$ norm $\norm{f}_4$ of a Littlewood polynomial $f$ (having all coefficients in $\{-1,+1\}$) of degree $n-1$ can grow with $n$. We consider such polynomials for odd square-free $n$, where $\phi(n)$ coefficients are determined by the Jacobi symbol, but the remaining coefficients can be freely chosen. When $n$ is prime, these polynomials have the smallest published asymptotic value of the normalised $L_4$ norm $\norm{f}_4/\norm{f}_2$ among all Littlewood polynomials, namely $(7/6)^{1/4}$. When $n$ is not prime, our results show that the normalised $L_4$ norm varies considerably according to the free choices of the coefficients and can even grow without bound. However, by suitably choosing these coefficients, the limit of the normalised $L_4$ norm can be made as small as the best published value $(7/6)^{1/4}$. 
\end{abstract}

\maketitle


\section{Introduction}
\label{sec:intro}

For real $\alpha\ge 1$, the $L_\alpha$ norm of a  polynomial $A\in\C[z]$ on the unit circle is given by
\[
\norm{A}_\alpha:=\bigg(\frac{1}{2\pi}\int_0^{2\pi}\bigabs{A(e^{i\theta})}^\alpha d\theta\bigg)^{1/\alpha}.
\]
The polynomial $A(z)=\sum_{j=0}^{n-1}a_jz^j$ is called a \emph{Littlewood polynomial} if $a_j\in\{-1,+1\}$ for each~$j$. In 1966, Littlewood~\cite[\S~6]{Littlewood1966} raised the question of how slowly the $L_4$ norm of a Littlewood polynomial of degree $n-1$ can grow with $n$. An equivalent question was posed by Turyn~\cite[p.~199]{Turyn1968} in a different context. Littlewood's question is closely related to other classical problems involving norms of Littlewood polynomials~\cite{Newman1960}, \cite{Erdos1962}, \cite{Littlewood1968}, \cite{Newman1990}, \cite{Beck1991}, \cite{Borwein2002a}.
\par
For a polynomial $A\in\C[z]$, a small $L_4$ norm corresponds to a large \emph{merit factor}, defined as
\[
F(A):=\frac{\norm{A}_2^4}{\norm{A}_4^4-\norm{A}_2^4}
\]
provided that the denominator is nonzero. This normalised measure appears natural since it often attains an integer value when the polynomial degree tends to infinity. Littlewood's question concerns the growth rate of $F(A)$ since $\norm{A}_2^4=n^2$ for every Littlewood polynomial of degree $n-1$. The determination of the largest possible merit factor of Littlewood polynomials of large degree is also of importance in the theory of communications, where Littlewood polynomials with large merit factor correspond to signals whose energy is very evenly distributed over frequency~\cite{Beenker1985}, and in theoretical physics, where Littlewood polynomials with largest merit factor correspond to the ground states of Bernasconi's Ising spin model~\cite{Bernasconi1987}.
\par
If $A$ is drawn uniformly from the set of Littlewood polynomials of degree $n-1$, then $F(A)\to1$ in probability as $n\to\infty$~\cite{Borwein2001a}. Littlewood~\cite{Littlewood1968} constructed a sequence of Littlewood polynomials with asymptotic merit factor~$3$. Since then Littlewood's question has been attacked by mathematicians, engineers, and physicists (see~\cite{Jedwab2005} for a survey of results and historical developments).
\par
Given a polynomial $A\in\C[z]$ of degree $n-1$ and real $r$, define the \emph{rotation} $A_r$ of $A$ by
\begin{equation}
A_r(z):=z^{-\lfloor nr\rfloor}A(z)\bmod{(z^n-1)}.   \label{eqn:def_rot}
\end{equation}
For odd $n$, let $\leg{\,\cdot\,}{n}$ be the Jacobi symbol (see~\cite{Apostol1976}, for example), and call
\[
J(z):=\sum_{j=1}^{n-1}\leg{j}{n}\,z^j
\]
the \emph{character polynomial} of degree $n-1$. For prime~$n$, this polynomial is known as the \emph{Fekete polynomial}, which has been studied extensively and whose asymptotic merit factor has been determined for all rotations (see~\cite{Montgomery1980}, \cite{Hoholdt1988}, \cite{Conrey2000}, \cite{Borwein2001b}, \cite{Borwein2002}, for example). Indeed, defining
\begin{equation}
f(r):=\begin{cases}
\dfrac{1}{\frac{1}{6}+8\left(\abs{r}-\frac{1}{4}\right)^2} & \mbox{for $-\frac{1}{2}<r\le\frac{1}{2}$}\\[3ex]
f(r+1) & \mbox{otherwise},
\end{cases}
\label{eqn:frdefn}
\end{equation}
the following result is known.
\begin{theorem}[{H{\o}holdt and Jensen~\cite{Hoholdt1988}}]
\label{thm:hj}
Let $p$ take values in an infinite set of odd primes, and let $r$ be real. Let $X=J+1$, where $J$ is the character polynomial of degree $p-1$. Then
\[
\lim_{p\to\infty}F(X_r)=f(r).
\]
\end{theorem}
\par
Borwein and Choi~\cite{Borwein2002} also calculated the exact, rather than the asymptotic, values of $F(X)$ and $F(X_{1/4})$ by refining the proof of Theorem~\ref{thm:hj}. The largest asymptotic merit factor occurring in Theorem~\ref{thm:hj} is~$6$. The polynomial $X$ of degree $p-1$ in Theorem~\ref{thm:hj} has been used to construct Littlewood polynomials of degree $2p-1$~\cite{Xiong2008} and $4p-1$~\cite{Schmidt2009} that also have asymptotic merit factor $6$, and the value~$6$ remains the largest published asymptotic merit factor for all sequences of Littlewood polynomials. H{\o}holdt and Jensen~\cite{Hoholdt1988} conjectured that no larger value is possible, although there are various contradicting opinions~\cite[p.~29]{Littlewood1968}, \cite{Golay1982}, \cite{Borwein2004}. In contrast, there are sequences of polynomials, not all of whose coefficients lie in $\{-1,+1\}$, for which the merit factor grows without bound as the degree increases~\cite[\S~6]{Littlewood1966}.
\par
In this paper we study the case when $n$ is square-free but not prime. The character polynomial $J$ of degree $n-1$ has $\phi(n)$ nonzero coefficients since $\leg{j}{n}=0$ exactly when $\gcd(j,n)>1$. Define
\[
\V_n:=\bigg\{\sum_{j=0}^{n-1}v_jz^j:v_j\in\{0,-1,+1\}~\text{and $v_j=0\Leftrightarrow\gcd(j,n)=1$}\bigg\}.
\]
The polynomial $J+V$ is then a Littlewood polynomial for each $V\in\V_n$, and we call $J+V$ a \emph{Littlewood completion} of~$J$. We wish to determine the choice of $V\in\V_n$ for each $n$ and the choice of $r$ that maximise the asymptotic merit factor of $J_r+V_r$. In the case when $n$ is prime, there are only two possible Littlewood completions of $J$, namely $J+1$ and $J-1$. Theorem~\ref{thm:hj} deals with $J+1$, and it is readily seen that the same result holds for $J-1$. However, for general~$n$ there are $2^{n-\phi(n)}$ possible Littlewood completions of $J$. The choice of the Littlewood completion and rotation that maximise the asymptotic merit factor is then by no means obvious and the analysis is considerably more difficult.


\section{Results}

Throughout this paper, we will use the following notation. For integer~$n>1$, we define $p_n$ to be the smallest prime factor of $n$ and, as usual, $\w(n)$ denotes the number of distinct prime factors of~$n$.
\par
As a starting point we establish the asymptotic merit factor of the character polynomial $J$ itself at all rotations.
\begin{theorem}
\label{thm:ternary}
Let $n$ take values only in an infinite set of odd square-free integers greater than~$1$, where
\begin{equation}
\frac{(\log n)^3}{p_n}\to 0 \quad\mbox{as $n\to\infty$},
\label{eqn:logn2p1}
\end{equation}
and let $r$ be real. Let $J$ be the character polynomial of degree $n-1$. Then
\[
\lim_{n\to\infty} F(J_r)=f(r).
\]
\end{theorem}
\par
We next examine the special Littlewood completion $J+V$ of~$J$ in which each nonzero coefficient of $V$ is chosen to be~$+1$.
\begin{theorem}
\label{thm:all-ones-mainres}
Let $n$ take values only in an infinite set of odd square-free integers greater than~$1$ and let $r$ be real. Let $J$ be the character polynomial of degree $n-1$ and define
\[
V(z)=\sum_{\stack{j=0}{\gcd(j,n)>1}}^{n-1}z^j.
\]
Then
\begin{equation}
\liminf_{n\to\infty}\frac{1}{F(J_r+V_r)}\ge\liminf_{n\to\infty}\frac{1}{F(J_r)}+\liminf_{n\to\infty}\frac{n}{2p_n^3}.   \label{eqn:liminf_all-ones}
\end{equation}
\end{theorem}
\par
Hence, if $p_n/n^{1/3}$ is bounded (which occurs for example if $\w(n)\ge 3$ for all sufficiently large~$n$), then
\[
\limsup_{n\to\infty}F(J_r+V_r)<\limsup_{n\to\infty}F(J_r),
\]
and if $p_n/n^{1/3}\to0$ (which occurs for example if $\w(n)\ge 4$ for all sufficiently large~$n$), then
\[
\lim_{n\to\infty}F(J_r+V_r)=0.
\]
Subject to the condition~\eqref{eqn:logn2p1}, we may replace $\liminf_{n\to\infty}1/F(J_r)$ in Theorem~\ref{thm:all-ones-mainres} by $1/f(r)$. Theorem~\ref{thm:all-ones-mainres} therefore shows that the asymptotic merit factor of $J_r+V_r$ can be strictly less than~$f(r)$ for all~$r$. This prompts the question of whether there is a choice of~$V$ for which the asymptotic merit factor of $J_r+V_r$ is \emph{greater} than $f(r)$ for some~$r$. However, we show that, subject to a mild condition on the growth rate of~$p_n$ relative to~$n$, there is no such~$V$.
\begin{theorem}
\label{thm:upper-bound-mainres}
Let $n$ take values only in an infinite set of odd square-free integers greater than~$1$, where
\begin{equation}
\frac{(\log n)^7}{p_n}\to 0 \quad\mbox{as $n\to\infty$},
\label{eqn:logn6p1-mainres}
\end{equation}
and let $r$ be real. Let $J$ be the character polynomial of degree $n-1$. Then
\[
\limsup_{n\to\infty}\,\max_{V\in\V_n}\,F(J_r+V_r)\le f(r).
\]
\end{theorem}
\par
We then ask whether the deterioration in asymptotic merit factor obtained in Theorem~\ref{thm:all-ones-mainres} for a specific choice of~$V$ is typical of Littlewood completions of $J$. We show it is not: subject to the same condition~\eqref{eqn:logn6p1-mainres} as in Theorem~\ref{thm:upper-bound-mainres}, for almost all choices of~$V$ we have $F(J_r+V_r)\sim f(r)$.
\begin{theorem}
\label{thm:almost-all-mainres}
Let $n$ take values only in an infinite set of odd square-free integers greater than~$1$, where
\begin{equation}
\frac{(\log n)^7}{p_n}\to 0 \quad\mbox{as $n\to\infty$},   \label{eqn:logn6p1-2-mainres}
\end{equation}
and let $r$ be real. Let $J$ be the character polynomial of degree $n-1$ and let $V$ be drawn uniformly from~$\V_n$. Then, as $n\to\infty$,
\[
F(J_r+V_r)\to f(r)\quad\mbox{in probability}.
\]
\end{theorem}
\par
In view of Theorem~\ref{thm:almost-all-mainres}, we wish to exhibit polynomials~$V\in\V_n$ satisfying $\lim_{n\to\infty} F(J_r+V_r)=f(r)$ under suitable conditions on the growth rate of~$p_n$ relative to~$n$. We present two such choices of polynomials~$V$. The first choice is given in the following theorem.
\begin{theorem}
\label{thm:product-mainres}
Let $n$ take values only in an infinite set of odd square-free integers greater than~$1$, where 
\begin{equation}
\frac{(\log n)^7}{p_n}\to 0 \quad\mbox{as $n\to\infty$},
\label{eqn:logn6p1-3-mainres}
\end{equation}
and let $r$ be real. Let $J$ be the character polynomial of degree $n-1$ and define
\begin{equation}
V(z)=\sum_{\stack{j=0}{\gcd(j,n)>1}}^{n-1}\bigleg{j}{\tfrac{n}{\gcd(j,n)}}\,z^j.   \label{eqn:LegV-mainres}
\end{equation}
Then
\[
\lim_{n\to\infty}F(J_r+V_r)=f(r).
\]
\end{theorem}
\par
The special case of Theorem~\ref{thm:product-mainres} when $\w(n)=1$ for all $n$ gives Theorem~\ref{thm:hj}.
\par
The second choice of polynomials~$V\in\V_n$ satisfying $\lim_{n\to\infty}F(J_r+V_r)=f(r)$ uses a more restrictive condition than \eqref{eqn:logn6p1-3-mainres} in Theorem~\ref{thm:product-mainres}, but applies to \emph{all} Littlewood completions.
\begin{theorem}
\label{thm:twoprimes-mainres}
Let $n$ take values only in an infinite set of odd square-free integers greater than~$1$, where
\begin{equation}
\frac{n^{1/3}}{p_n}\to0 \quad\mbox{as $n\to\infty$},
\label{eqn:cond-2primes-mainres}
\end{equation}
and let $r$ be real. Let $J$ be the character polynomial of degree $n-1$. Then
\[
\lim_{n\to\infty}\,\max_{V\in\V_n}\,F(J_r+V_r)=\lim_{n\to\infty}\,\min_{V\in\V_n}\,F(J_r+V_r)=f(r).
\]
\end{theorem}
\par
The condition~\eqref{eqn:cond-2primes-mainres} is essentially  the least restrictive condition under which Theorem~\ref{thm:twoprimes-mainres} holds: for if $\liminf_{n\to\infty}n^{1/3}/p_n>0$, then by Theorem~\ref{thm:all-ones-mainres} the conclusion of Theorem~\ref{thm:twoprimes-mainres} fails for at least one Littlewood completion~$J+V$; but otherwise $\liminf_{n\to\infty}n^{1/3}/p_n=0$, and then the infinite set in which~$n$ takes values contains a subset satisfying the condition~\eqref{eqn:cond-2primes-mainres}.
\par
We shall prove Theorems~\ref{thm:ternary} to~\ref{thm:twoprimes-mainres} in Sections~\ref{sec:ternary} to~\ref{sec:two-primes}, respectively. Our results provide a comprehensive analysis of the $2^{n-\phi(n)}$ Littlewood completions of the character polynomial $J$ of degree $n-1$, and significantly enlarge the set of explicitly defined sequences of Littlewood polynomials whose asymptotic merit factor equals the current best known value $6$.
\par
We close this section with a brief review of related work. Jensen, Jensen and H{\o}holdt \cite{Jensen1991} gave the asymptotic merit factor of two Littlewood completions $J+V$ of $J$ in the case that $\w(n)=2$ for all $n$. For one of these completions, the polynomial $V$ coincides with~\eqref{eqn:LegV-mainres}; for the other, writing $n=pq$ for primes $p,q$ satisfying $p>q$, the polynomial $V$ is given by
\[
V(z)=\sum_{j=0}^{p-1}z^{jq}-\sum_{j=1}^{q-1}z^{jp}.
\]
The results of~\cite{Jensen1991} for both of these Littlewood completions are special cases of Theorem~\ref{thm:twoprimes-mainres}. The authors of~\cite{Jensen1991} also stated that the conclusion of Theorem~\ref{thm:product-mainres} holds when $\w(n)$ is fixed, but did not give a proof or specify conditions on the growth rate of $p_n$.
\par
Motivated by the results of~\cite{Jensen1991}, Borwein and Choi~\cite{Borwein2001} proved a result that gives the same conclusion as Theorem~\ref{thm:ternary} under the more restrictive condition $n^\e/p_n\to 0$ for some fixed $\e>0$. The authors of \cite{Borwein2001} remarked that
\begin{quote}
``the merit factors [of the polynomials $J_{1/4}$ as $n\to\infty$] approach~6 which is conjectured by some to be best possible~\cite{Golay1983},''
\end{quote} 
and that their result
\begin{quote}
``should be compared with the results of T.~H{\o}holdt, H.~Jensen and J.~Jensen in~\cite{Jensen1991}. They showed that the same asymptotic formula but a weaker error term $O\Big(\frac{(p+q)^5 \log^4 N}{N^3}\Big)$ for the special case $N = pq$. So we generalize their result to $N=p_1p_2\dots p_r$ and also improve the error term.''
\end{quote}
However, the authors of \cite{Borwein2001} did not take into account the crucial distinction between the polynomial $J$ of degree $n-1$ and its $2^{n-\phi(n)}$ Littlewood completions. Indeed, Theorem~\ref{thm:all-ones-mainres} shows that there is a sequence of Littlewood completions of $J$ whose asymptotic merit factor at every rotation~$r$ drops to zero. Therefore the result of \cite{Borwein2001} cannot be considered a generalisation of the results of~\cite{Jensen1991}, and the comparison given in \cite{Borwein2001} with the conjecture of \cite{Golay1983} (which applies only to Littlewood polynomials) is misplaced.
\par
T.~Xiong and J.~I.~Hall have kindly supplied us with two preprints of their recent independent work. In the first preprint, now published as~\cite{Xiong2011}, they obtain the same asymptotic form as in Theorem~\ref{thm:twoprimes-mainres}, subject to the more restrictive condition that $(n\log n)^{2/5}/p_n\to 0$. In the second preprint~\cite{Xiong2011a}, they show that a previously unspecified Littlewood completion satisfies $\lim_{n\to\infty} F(J_r+V_r)=f(r)$ when $\w(n)$ is fixed.


\section{Preliminary Results}
\label{sec:defn}

In this section we introduce some notation and give some auxiliary results. Throughout the paper, $\z_m$ denotes the primitive $m$th root of unity
\[
\z_m:=e^{2\pi i/m}.
\]
We next derive some elementary bounds on the functions $\w(n)$ and $\phi(n)$. The number of distinct prime factors $\w(n)$ of $n$ can be trivially bounded by
\begin{equation}
\w(n)\le \log n\quad\mbox{for $n>2$ and $n\ne 6$}.    \label{eqn:wnlogn}
\end{equation}
Since $\phi(n)/n=\prod_{p\mid n}(1-1/p)$, where the product is over the prime factors of~$n$, the totient function $\phi(n)$ then satisfies
\begin{align*}
\frac{\phi(n)}{n}&\ge\left(1-\frac{1}{p_n}\right)^{\w(n)}\\
&\ge 1-\frac{\w(n)}{p_n}\\
&\ge 1-\frac{\log n}{p_n}\quad\mbox{for $n>2$ and $n\ne 6$},
\end{align*}
so we can estimate its growth rate as
\begin{equation}
\phi(n)=n\left(1+O(p_n^{-1}\log n)\right)\quad\mbox{as $n\to\infty$.}   \label{eqn:growth-phi}
\end{equation}
For convenience, we define the~\emph{cototient function} to be 
\[
\psi(n):=n-\phi(n).
\]
It follows that
\begin{align}
\frac{\psi(n)}{n}&\le\frac{\w(n)}{p_n}   \label{eqn:psi-omega}\\
&\le\frac{\log n}{p_n}  \quad\mbox{for $n>2$ and $n\ne 6$}   \label{eqn:psi-dmultiple}
\end{align}
and therefore
\begin{equation}
\psi(n)=O(p_n^{-1}n\log n)\quad\mbox{as $n\to\infty$}.   \label{eqn:growth-psi}
\end{equation}
\par
We shall need the following evaluation of Ramanujan's sum (see~\cite[Thm.~272]{Hardy1979}, for example).
\begin{lemma}
\label{lem:Ramanujan}
For integer $u$ and positive square-free integer $n$, we have
\[
\sum_{\stack{j=0}{\gcd(j,n)=1}}^{n-1}\z_n^{ju}=\mu\left(\frac{n}{\gcd(u,n)}\right)\,\phi\big(\gcd(u,n)\big),
\]
where $\mu$ is the M{\"o}bius function.
\end{lemma}
\par
We also require the following evaluation of a Gauss sum
involving the Jacobi symbol.
\begin{lemma}
\label{lem:gauss-sum}
Let $m$ be a positive odd square-free integer. Then for integer~$j$, 
\[
\sum_{\ell=0}^{m-1}\leg{\ell}{m}\, \z_m^{j\ell}=i^{(m-1)^2/4}\, \leg{j}{m}\, m^{1/2}.
\]
\end{lemma}
\par
The case $\gcd(j,m)=1$ of Lemma~\ref{lem:gauss-sum} is given by Thm.~1.5.2 and Ch.~1, Problem~24 of~\cite{Berndt1998}, for example. The case $\gcd(j,m)>1$ then follows by application of Parseval's identity.
\par
Now let $n$ be an odd square-free integer and let $J$ be the character polynomial of degree $n-1$. Lemma~\ref{lem:gauss-sum} with $m=n$ implies that, for integer~$j$,
\begin{equation}
J(\z_n^j)=i^{(n-1)^2/4}\, \leg{j}{n}\, n^{1/2}.   \label{eqn:magTe}
\end{equation}
Given a polynomial $A$ of degree $n-1$, then by the definition~\eqref{eqn:def_rot} of the rotation $A_r$, we have for integer~$j$
\begin{equation}
A_r(\z_n^j)=\z_n^{-j\lfloor nr\rfloor}A(\z_n^j)   \label{eqn:QA-shift}
\end{equation}
and therefore,
\begin{equation}
J_r(\z_n^j)=i^{(n-1)^2/4}\,\z_n^{-j\lfloor nr\rfloor}\,\leg{j}{n}\, n^{1/2}.   \label{eqn:magTe_rot}
\end{equation}
\par
We shall need the following bound for the magnitude of a polynomial of degree $n-1$ over~$\C$ on the unit circle in terms of its values at the $n$th roots of unity.
\begin{lemma}
\label{lem:poly-bound}
Let $A\in\C[z]$ have degree at most $n-1$ for $n>2$. Then
\[
\max_{\abs{z}=1}\;\abs{A(z)}\le(2\log n)\max_{0\le k<n}\,\abs{A(\z_n^k)}.
\]
\end{lemma}
\begin{proof}
By bounding the coefficients that occur in the Lagrange interpolation of $A$ from its evaluations at the $n$th roots of unity, it can be shown that
\[
\max_{\abs{z}=1}\;\abs{A(z)}\le c(n)\,\max_{0\le k<n}\,\abs{A(\z_n^k)},
\]
where $c(n)=1+(1/n)\sum_{j=1}^{n-1}1/\sin(\frac{\pi j}{2n})$ (see~\cite[Appendix]{Paterson2000}, for example). Since $c(n)<1+\sum_{j=1}^{n-1}1/j$ and $\sum_{j=2}^{n-1}1/j<\log n$, the lemma holds for $n>7$. By direct verification we also have $c(n)\le 2\log n$ for $3\le n\le 7$.
\end{proof}
\par
Using~\eqref{eqn:magTe_rot}, Lemma~\ref{lem:poly-bound} gives
\begin{equation}
\max_{\abs{z}=1}\;\abs{J_r(z)}\le 2n^{1/2}\log n.   \label{eqn:bound_J_unitcircle}
\end{equation}
\par
We next prove our main tool for comparing the asymptotic merit factor of $J$ with that of a Littlewood completion~$J+V$.
\begin{proposition}
\label{prop:L4-sum}
Let $n>1$ be an odd square-free integer, and let $r$ be real. Then all Littlewood completions~$J+V$ of the character polynomial~$J$ of degree~$n-1$ satisfy
\[
\Biggabs{\frac{1}{F(J_r+V_r)}-\left(\frac{\phi(n)}{n}\right)^2\frac{1}{F(J_r)}-\frac{\norm{V_r}_4^4}{n^2}}<8\,p_n^{-1/2}n^{-1}(\log n)^{3/2}\norm{V_r}_4^2+58p_n^{-1/2}(\log n)^{7/2}.
\]
\end{proposition}
\par
In the application of Proposition~\ref{prop:L4-sum} it is sometimes useful to further bound $\norm{V_r}_4^4$ as
\begin{equation}
\norm{V_r}_4^4\le [\psi(n)]^3,   \label{eqn:norm_V_r_bound}
\end{equation}
which follows from $\norm{V_r}_2^2=\psi(n)$ and the simple inequality
\begin{equation}
\norm{A}_4^4\le \norm{A}_2^2\;\max_{\abs{z}=1}\,\abs{A(z)}^2\quad\text{for all $A\in\C[z]$}.   \label{eqn:L4L2max}
\end{equation}
\begin{proof}[Proof of Proposition~\ref{prop:L4-sum}]
Let $V\in\V_n$ and let
\begin{align}
\beta(n)&:=\Biggabs{\frac{1}{F(J_r+V_r)}-\left(\frac{\phi(n)}{n}\right)^2\!\!\frac{1}{F(J_r)}-\frac{\norm{V_r}_4^4}{n^2}}.   \nonumber \\
\intertext{Since $\norm{J_r}_2^2=\phi(n)$ and $\norm{J_r+V_r}_2^2=n$, we have by the definition of the merit factor}
\beta(n)&=\biggabs{\frac{1}{n^2}\,\Big(\norm{J_r+V_r}_4^4-\norm{J_r}_4^4-\norm{V_r}_4^4\Big)+\left(\frac{\phi(n)}{n}\right)^2-1}.   \label{eqn:upsilon}
\end{align}
Since
\[
\biggabs{\left(\frac{\phi(n)}{n}\right)^2 -1}=\frac{1}{n^2}\Bigabs{\big(\phi(n)+n\big)\big(\phi(n)-n\big)}<\frac{2\psi(n)}{n}
\]
by the trivial inequality $\phi(n)+n<2n$, it follows from \eqref{eqn:upsilon} that
\begin{equation}
\beta(n)<\biggabs{\frac{1}{n^2}\,\Big(\norm{J_r+V_r}_4^4-\norm{J_r}_4^4-\norm{V_r}_4^4\Big)}+\frac{2\psi(n)}{n}.   \label{eqn:gn}
\end{equation}
Now for $a,b\in\C$, by expanding $\abs{a+b}^4$, we get the inequality
\[
\Bigabs{\abs{a+b}^4-\abs{a}^4-\abs{b}^4}\le 4\,\abs{a}^3\cdot\abs{b}+6\,\abs{a}^2\cdot\abs{b}^2+4\,\abs{a}\cdot\abs{b}^3.
\]
Use~\eqref{eqn:bound_J_unitcircle} and the definition of the $L_\alpha$ norm to conclude from~\eqref{eqn:gn} that
\begin{equation}
\beta(n)<\frac{32(\log n)^3}{n^{1/2}}\,\norm{V_r}_1+\frac{24(\log n)^2}{n}\,\norm{V_r}_2^2+\frac{8\log n}{n^{3/2}}\,\norm{V_r}_3^3+\frac{2\psi(n)}{n}.   \label{eqn:beta-norms}
\end{equation}
We have $\norm{V_r}_2^2=\psi(n)$. By the Cauchy-Schwarz inequality,
\[
\norm{V_r}_{m+1}^{m+1}\le \norm{V_r}_2\, \bigg(\frac{1}{2\pi}\int_0^{2\pi}\bigabs{V_r(e^{i\theta})}^{2m}\,d\theta\bigg)^{1/2}.
\]
Hence $\norm{V_r}_1\le[\psi(n)]^{1/2}$ and $\norm{V_r}_3^3\le [\psi(n)]^{1/2}\,\norm{V_r}_4^2$, by taking $m=0$ and $m=2$, respectively. Therefore, using~\eqref{eqn:psi-dmultiple} to bound $\psi(n)$, we find from~\eqref{eqn:beta-norms} that
\begin{align*}
\beta(n)&<32p_n^{-1/2}(\log n)^{7/2}+24p_n^{-1}(\log n)^3+8p_n^{-1/2}n^{-1}(\log n)^{3/2}\norm{V_r}_4^2+2p_n^{-1}\log n\\
&<8p_n^{-1/2}n^{-1}(\log n)^{3/2}\norm{V_r}_4^2+(32+24+2)p_n^{-1/2}(\log n)^{7/2}
\end{align*}
since $n>2$.
\end{proof}


\section{Proof of Theorem~\ref{thm:ternary}}
\label{sec:ternary}

In this section we determine the asymptotic merit factor of the character polynomial $J$ of degree $n-1$ at all rotations, proving Theorem~\ref{thm:ternary}.
\par
We need the following evaluation of a character sum.
\begin{lemma}
\label{lem:RX}
Let $n$ be a positive odd square-free integer. Then, for integer $u$,
\[
\sum_{j=0}^{n-1}\leg{j}{n}\leg{j+u}{n}=\mu\left(\frac{n}{\gcd(u,n)}\right)\,\phi\big(\gcd(u,n)\big).
\]
\end{lemma}
\begin{proof}
Given a polynomial $A(z)=\sum_{j=0}^{n-1}a_jz^j$ with real-valued coefficients, it is readily verified that
\[
\sum_{j=0}^{n-1}a_ja_{(j+u)\bmod n}=\frac{1}{n}\sum_{j=0}^{n-1}\abs{A(\z_n^j)}^2\,\z_n^{ju}.
\]
Applying this relation to the character polynomial $J$ of degree $n-1$ and using~\eqref{eqn:magTe}, then gives 
\[
\sum_{j=0}^{n-1}\leg{j}{n}\leg{j+u}{n}=\sum_{\stack{j=0}{\gcd(j,n)=1}}^{n-1}\z_n^{ju},
\]
which is Ramanujan's sum. The result now follows from Lemma~\ref{lem:Ramanujan}.
\end{proof}
\par
H{\o}holdt and Jensen \cite{Hoholdt1988} introduced a method for
calculating the merit factor of a polynomial of even degree. The following result summarises their method (and occurs as a special case of the slightly more general result of~\cite[Lem.~10]{Schmidt2009}).
\begin{lemma}
\label{lem:sum-ccf}
Let $A\in\R[z]$ be a polynomial of even degree~$n-1$. Define
\begin{equation}
\Lambda_A(j,k,\ell):=\sum_{a=0}^{n-1}A(\z_n^a) \overline{A(\z_n^{a+j})} A(\z_n^{a+k}) \overline{A(\z_n^{a+\ell})} \quad \mbox{for integer $j$, $k$, $\ell$}.   \label{eqn:Lambda}
\end{equation}
Then
\begin{equation}
\frac{\norm{A}_4^4}{n^2}
 = \frac{2n^2+1}{3n^5} \, \Lambda_A(0,0,0) + B + C + D,
\label{eqn:Tst}
\end{equation}
where
\begin{align*}
B &= \frac{2}{n^5} \sum_{k=1}^{n-1} \frac{\Lambda_A(0,0,k) + \z_n^k\,\overline{\Lambda_A(0,0,k)}} {(1-\z_n^k)^2}
     \cdot (1+\z_n^k), \\[1ex]
C & = -\frac{2}{n^5} \sum_{\stack{1 \le k, \ell<n}{k\ne\ell}}
\frac{4\,\z_n^k\,\,\Lambda_A(0,k,\ell)+\Lambda_A(k,0,\ell) + \z_n^k\z_n^\ell \,\,\overline{\Lambda_A(k,0,\ell)}}{(1-\z_n^k)(1-\z_n^{\ell})}, \\[1ex]
D &= \frac{4}{n^5} \sum_{k=1}^{n-1} 
\frac{2\Lambda_A(0,k,k) + \z_n^{-k} \,\Lambda_A(k,0,k)}{\abs{1-\z_n^k}^2}.
\end{align*}
\end{lemma}
\par
We are now ready to calculate the asymptotic merit factor of the character polynomial at all rotations.
\begin{proof}[Proof of Theorem~\ref{thm:ternary}.]
Without loss of generality, we may assume that $-\frac{1}{2}<r\le\frac{1}{2}$. Since $\norm{J_r}_2^2=\phi(n)$, we have by the definition of the merit factor
\[
\frac{1}{F(J_r)}=\left(\frac{n}{\phi(n)}\right)^2 \left(\frac{\norm{J_r}_4^4}{n^2}\right)-1.
\]
We claim that
\begin{equation}
\frac{\norm{J_r}_4^4}{n^2}=1+\frac{1}{f(r)}+O\big(p_n^{-1}(\log n)^3\big),   \label{eqn:S4Tr-claim}
\end{equation}
which then implies the desired result using the condition~\eqref{eqn:logn2p1} and the growth rate~\eqref{eqn:growth-phi} of $\phi(n)$.
\par
It remains to prove the claim~\eqref{eqn:S4Tr-claim}. Write $R:=\lfloor{nr}\rfloor$. We apply Lemma~\ref{lem:sum-ccf} to the polynomial~$J_r$ to give an expression for $\norm{J_r}_4^4/n^2$. We find the asymptotic form of this expression, evaluating the term involving $\Lambda_{J_r}(0,0,0)$ and the sum~$D$, and bounding the sums $B$ and~$C$.
\par
Using~\eqref{eqn:magTe_rot} and~\eqref{eqn:Lambda}, we have
\begin{equation}
\Lambda_{J_r}(j,k,\ell)=\z_n^{R(j-k+\ell)}\cdot n^2 \sum_{a=0}^{n-1}\leg{a}{n}\leg{a+j}{n}\leg{a+k}{n}\leg{a+\ell}{n}.   \label{eqn:Lambda-leg}
\end{equation}
\par
\begin{description}[font=\itshape]
\item[The term involving $\Lambda_{J_r}(0,0,0)$.]
By~\eqref{eqn:Lambda-leg} we have
\begin{align}
\frac{2n^2+1}{3n^5}\,\Lambda_{J_r}(0,0,0)&=\frac{2n^2+1}{3n^5}\,n^2\,\phi(n)   \nonumber \\[1ex]
&=\frac{2}{3}+O\big(p_n^{-1}\log n\big)   \label{eqn:Asum}
\end{align}
from the growth rate~\eqref{eqn:growth-phi} of~$\phi(n)$.
\par
\item[The sum $D$.]
By~\eqref{eqn:Lambda-leg}, for each $k$ we have
\[
\phi(n)-\psi(n)\le\frac{1}{n^2}\,\Lambda_{J_r}(0,k,k)\le \phi(n).
\]
From the growth rate~\eqref{eqn:growth-phi} of~$\phi(n)$ 
and the growth rate~\eqref{eqn:growth-psi} of~$\psi(n)$ we then obtain
\begin{align*}
\Lambda_{J_r}(0,k,k)&=n^3\big[1+O(p_n^{-1}\log n)\big]
\intertext{and similarly}
\Lambda_{J_r}(k,0,k)&=\z_n^{2Rk}\cdot n^3\big[1+O(p_n^{-1}\log n)\big].
\end{align*}
The sum $D$ then becomes
\begin{equation}
\label{eqn:Dsum1}
D=\frac{4}{n^2}\Big[1+O(p_n^{-1}\log n)\Big]\sum_{k=1}^{n-1}\frac{\,\,2+\z_n^{(2R-1)k}}{\abs{1-\z_n^k}^2}.
\end{equation}
We will evaluate the summation in \eqref{eqn:Dsum1} by using the identity 
\begin{equation}
\label{eqn:exp-sum-identity}
\sum_{k=1}^{n-1} \frac{\z_n^{jk}}{\abs{1-\z_n^k}^2}
=\frac{n^2}{2}\left(\frac{\abs{j}}{n}-\frac{1}{2}\right)^2-\frac{n^2+2}{24}\quad \mbox{for integer $j$ satisfying $\abs{j} \le n$}
\end{equation}
(see, \cite[p.~621]{Jensen1991}, for example). The assumption $-\tfrac{1}{2}<r\le\frac{1}{2}$ implies that $-n < 2R-1 < n$ for all sufficiently large $n$. We can therefore use~\eqref{eqn:exp-sum-identity} to evaluate the summation
in~\eqref{eqn:Dsum1} for all sufficiently large $n$, so that we have
\[
D=\frac{4}{n^2}\big[1+O(p_n^{-1}\log n)\big]\bigg[\frac{n^2}{2}\left(\frac{\abs{2R-1}}{n}-\frac{1}{2}\right)^2+\frac{n^2-2}{8}\bigg].
\]
By definition of $R$, we have $R=nr+O(1)$. We then find that
\begin{equation}
D=\tfrac{1}{2}+8\left(\abs{r}-\tfrac{1}{4}\right)^2+O(p_n^{-1}\log n).
\label{eqn:Dsum}
\end{equation}

\item[The sum $B$.] We bound the sum $B$ via
\begin{align}
\abs{B}&\le\frac{2}{n^5}\sum_{k=1}^{n-1}\frac{\,4 \,\,\bigabs{\Lambda_{J_r}(0,0,k)}}{\abs{1-\z_n^k}^2}   \nonumber \\
&=\frac{8}{n^5}\sum_{k=1}^{n-1}\frac{n^2}{\abs{1-\z_n^k}^2}\,\Biggabs{\sum_{a=0}^{n-1}\leg{a}{n}\leg{a+k}{n}}   \label{eqn:Bsum1}
\end{align}
by~\eqref{eqn:Lambda-leg}. But from Lemma~\ref{lem:RX} we know that
\begin{equation}
\Biggabs{\sum_{a=0}^{n-1}\leg{a}{n}\leg{a+k}{n}}\le\phi(p_n^{-1}n)<\frac{n}{p_n}\quad\mbox{for $k\not\equiv 0\pmod n$}.   \label{eqn:RT-bound}
\end{equation}
Substitution in~\eqref{eqn:Bsum1} gives
\begin{align}
\abs{B}&<\frac{8}{n^2\,p_n}\sum_{k=1}^{n-1}\frac{1}{\abs{1-\z_n^k}^2},   \nonumber \\
&=\frac{2(n^2-1)}{3n^2\,p_n}  \nonumber
\intertext{from~\eqref{eqn:exp-sum-identity}. Hence,}
B &=O(p_n^{-1}).   \label{eqn:Bsum}
\end{align}

\item[The sum $C$.] 
Since $\abs{\Lambda_{J_r}(0,k,\ell)}=\abs{\Lambda_{J_r}(k,0,\ell)}$ by~\eqref{eqn:Lambda-leg}, we can bound the sum~$C$ via
\begin{align}
\abs{C}&\le\frac{2}{n^5}\sum_{\stack{1\le k,\ell<n}{k\ne\ell}}\frac{6\,\,\bigabs{\Lambda_{J_r}(0,k,\ell)}}{\abs{1-\z_n^k}\cdot \abs{1-\z_n^{\ell}}}.   \label{eqn:C-bound}
\end{align}
Now from~\eqref{eqn:Lambda-leg} we have
\begin{align*}
\frac{1}{n^2}\,\bigabs{\Lambda_{J_r}(0,k,\ell)}
&=\Biggabs{\sum_{a=0}^{n-1}\leg{a+k}{n}\leg{a+\ell}{n}
\,\,\,-\!\!\!\sum_{\stack{a=0}{\gcd(a,n)>1}}^{n-1}\leg{a+k}{n}\leg{a+\ell}{n}}\\
&\le \Biggabs{\sum_{a=0}^{n-1}\leg{a}{n}\leg{a+\ell-k}{n}}+\psi(n)\\
&<\frac{n}{p_n}+\psi(n)\quad\mbox{for $k\not\equiv\ell\pmod n$}
\end{align*}
by \eqref{eqn:RT-bound}. Substitution in~\eqref{eqn:C-bound} then gives
\begin{align*}
\abs{C}&<\frac{12}{n^3}\bigg(\frac{n}{p_n}+\psi(n)\bigg)\sum_{\stack{1\le k,\ell<n}{k\ne\ell}}\frac{1}{\abs{1-\z_n^k}\cdot \abs{1-\z_n^{\ell}}}\\
&<\frac{12}{n^3}\bigg(\frac{n}{p_n}+\psi(n)\bigg)\bigg(\sum_{k=1}^{n-1}\frac{1}{\abs{1-\z_n^k}}\bigg)^2\\
&\le \frac{12(\log n)^2}{n}\left(\frac{n}{p_n}+\psi(n)\right)
\end{align*}
since $\sum_{k=1}^{n-1} 1/\abs{1-\z_n^k}\le n\log n$ (see \cite[p.~163]{Hoholdt1988}, for example). 
Then from the growth rate~\eqref{eqn:growth-psi} of~$\psi(n)$ we obtain
\begin{equation}
C=O(p_n^{-1}(\log n)^3).
\label{eqn:Csum}
\end{equation}
\end{description}

The claim~\eqref{eqn:S4Tr-claim} now follows by substituting the 
asymptotic forms 
\eqref{eqn:Asum}, \eqref{eqn:Dsum}, \eqref{eqn:Bsum}, and~\eqref{eqn:Csum}
in~(\ref{eqn:Tst}), and then using the definition~\eqref{eqn:frdefn} of~$f$.
\end{proof}


\section{Proof of Theorem~\ref{thm:all-ones-mainres}}
\label{sec:all-ones}

By Proposition~\ref{prop:L4-sum}, we have
\[
\frac{1}{F(J_r+V_r)}>\left(\frac{\phi(n)}{n}\right)^2\frac{1}{F(J_r)}+\delta(n),
\]
where
\begin{align}
\delta(n)
&=\frac{1}{n^2}\,\norm{V_r}_4^4
-8p_n^{-1/2}n^{-1}(\log n)^{3/2}\,\norm{V_r}_4^2 
-58p_n^{-1/2} (\log n)^{7/2}   \nonumber \\
&=\frac{1}{n^2}\,\norm{V_r}_4^4
+O\big(p_n^{-2}n^{1/2}(\log n)^3\big)
+O\big(p_n^{-1/2}(\log n)^{7/2}\big),   \label{eqn:delta}
\end{align}
using the upper bound~\eqref{eqn:norm_V_r_bound} for $\norm{V_r}_4^4$ and the upper bound~\eqref{eqn:psi-dmultiple} for $\psi(n)$. Therefore
\begin{equation}
\liminf_{n\to\infty}\frac{1}{F(J_r+V_r)}\ge\liminf_{n\to\infty}\left[\left(\frac{\phi(n)}{n}\right)^2\frac{1}{F(J_r)}\right]+\liminf_{n\to\infty}\delta(n).   \label{eqn:liminf_F_delta}
\end{equation}
We next derive a lower bound for the term $\norm{V_r}_4^4/n^2$ in~\eqref{eqn:delta}, giving an asymptotic lower bound for~$\delta(n)$. For a polynomial $A\in\C[z]$ of degree at most $n-1$, we have the identity
\[
\norm{A}_4^4=\frac{1}{2n}\bigg(\sum_{j=0}^{n-1}\abs{A(\z_n^j)}^4+\sum_{j=0}^{n-1}\abs{A(-\z_n^j)}^4\bigg)
\]
(see~\cite{Hoholdt1988}, for example), which gives the inequality
\begin{align}
\frac{1}{n^2}\,\norm{V_r}_4^4&\ge\frac{1}{2n^3}\sum_{j=0}^{n-1}\abs{V_r(\z_n^j)}^4.   \nonumber 
\intertext{Restrict the summation to the set $U=\{ \frac{n}{p_n},2\frac{n}{p_n},\dots,(p_n-1)\frac{n}{p_n}\}$ and use~\eqref{eqn:QA-shift} to obtain}
\frac{1}{n^2}\,\norm{V_r}_4^4&\ge \frac{1}{2n^3}\sum_{u\in U}\abs{V(\z_n^u)}^4.  \label{eqn:S4V-lb}
\end{align}
Now let $u\in U$. From the definition of $V$ we have
\begin{align*}
V(\z_n^u)&=\sum_{\stack{j=0}{\gcd(j,n)>1}}^{n-1}\z_n^{ju}\\
&=\sum_{j=0}^{n-1}\z_n^{ju} \,\,\, -\!\!\!\sum_{\stack{j=0}{\gcd(j,n)=1}}^{n-1}\z_n^{ju}.
\intertext{The first sum evaluates to $0$ because $\z_n^u\ne 1$. The second sum is Ramanujan's sum, and using $\gcd(u,n)=p_n^{-1}n$ in Lemma~\ref{lem:Ramanujan}, we get}
V(\z_n^u)&=\phi\big(p_n^{-1}n\big)=\frac{\phi(n)}{p_n-1}.
\end{align*}
Substitution in \eqref{eqn:S4V-lb} then gives the desired lower bound
\begin{align*}
\frac{1}{n^2}\,\norm{V_r}_4^4&\ge\frac{1}{2n^3}(p_n-1)\left(\frac{\phi(n)}{p_n-1}\right)^4\\
&>\frac{n}{2p_n^3}\left(\frac{\phi(n)}{n} \right)^4.
\end{align*}
By substituting this lower bound in~\eqref{eqn:delta} we find that
\begin{align}
\delta(n)&>\frac{n}{2p_n^3} \left(\frac{\phi(n)}{n}\right)^4
+O\big(p_n^{-2}n^{1/2}(\log n)^3\big)
+O\big(p_n^{-1/2}(\log n)^{7/2}\big),   \label{eqn:deltaLB1} 
\intertext{or equivalently}
\delta(n)&>\frac{n}{2p_n^3}\left[\left(\frac{\phi(n)}{n}\right)^4 
+O\big(p_n n^{-1/2}(\log n)^3\big)
+O\big(p_n^{5/2}n^{-1}(\log n)^{7/2}\big)\right].   \label{eqn:deltaLB2} 
\end{align}
\par
To complete the proof, partition the infinite set $N$, in which $n$ takes values, into subsets $N_1$, $N_2$ defined by
\[
n\in
\begin{cases} N_1 & \mbox{if $p_n\le n^{2/7}$} \\[0.5ex]
              N_2 & \mbox{if $p_n>   n^{2/7}$},
\end{cases}
\]
at least one of which is infinite. First suppose that $N_1$ is infinite and let $n$ take values only in $N_1$. Then
\[
p_nn^{-1/2}(\log n)^3\le n^{-3/14}(\log n)^3\to0
\]
and
\[
p_n^{5/2}n^{-1}(\log n)^{7/2}\le n^{-2/7}(\log n)^{7/2}\to0,
\]
so that by \eqref{eqn:deltaLB2} we obtain
\[
\liminf_{n\to\infty}\delta(n)\ge\liminf_{n\to\infty}\bigg[\frac{n}{2p_n^3}\left(\frac{\phi(n)}{n}\right)^4\bigg].
\]
Choose some $\e$ satisfying $0<\epsilon<1/28$. Since $\phi(n)/n^{1-\e}\to\infty$ (see~\cite[Thm.~327]{Hardy1979}, for example), we have
\[
\liminf_{n\to\infty} \delta(n)\ge\liminf_{n\to\infty}\frac{n^{1-4\epsilon}}{2p_n^3}\ge\frac{1}{2}\liminf_{n\to\infty}\,n^{1/7-4\epsilon}=\infty,
\]
so that by \eqref{eqn:liminf_F_delta},
\[
\liminf_{n\to\infty}\frac{1}{F(J_r+V_r)}=\infty.
\]
This verifies the claim~\eqref{eqn:liminf_all-ones} of the theorem when $n\in N_1$ since $p_n\le n^{2/7}$ for all $n\in N_1$.
\par
Now suppose that $N_2$ is infinite and let $n$ take values only in $N_2$. Then
\[
p_n^{-2} n^{1/2}(\log n)^3 < n^{-1/14}(\log n)^3\to0
\]
and
\[
p_n^{-1/2} (\log n)^{7/2} < n^{-1/7} (\log n)^{7/2}\to0,
\]
so that by \eqref{eqn:deltaLB1} we obtain
\[
\liminf_{n\to\infty}\delta(n)\ge\liminf_{n\to\infty}\bigg[\frac{n}{2p_n^3}\left(\frac{\phi(n)}{n} \right)^4\bigg].
\]
From the growth rate~\eqref{eqn:growth-phi} of $\phi(n)$ and~\eqref{eqn:liminf_F_delta} we then conclude that the claim~\eqref{eqn:liminf_all-ones} of the theorem holds when $n\in N_2$. Therefore it holds when $n\in N_1\cup N_2=N$, which completes the proof.\qed


\section{Proof of Theorem~\ref{thm:upper-bound-mainres}}
\label{sec:upper-bound}

The structure of the proof is broadly similar to that of Theorem~\ref{thm:all-ones-mainres}, except that we now use the condition \eqref{eqn:logn6p1-mainres} to control the term $\norm{V_r}_4^4$ for $V\in\V_n$. Application of Proposition~\ref{prop:L4-sum} gives, for each $V\in\V_n$,
\[
\frac{1}{F(J_r+V_r)}>\left(\frac{\phi(n)}{n}\right)^2\frac{1}{F(J_r)}+\delta(n),
\]
where
\begin{equation}
\delta(n)=\frac{1}{n^2}\,\norm{V_r}_4^4-8p_n^{-1/2}n^{-1}(\log n)^{3/2}\norm{V_r}_4^2-58p_n^{-1/2} (\log n)^{7/2}.
\label{eqn:deltaS4}
\end{equation}
We then find from the growth rate~\eqref{eqn:growth-phi} of~$\phi(n)$, using the condition \eqref{eqn:logn6p1-mainres}, that
\begin{equation}
\liminf_{n\to\infty}\min_{V\in\V_n}\frac{1}{F(J_r+V_r)} \ge 
 \liminf_{n\to\infty}\frac{1}{F(J_r)} 
 +\liminf_{n\to\infty}\delta(n).   \label{eqn:liminfLB}
\end{equation}
We claim that 
\begin{equation}
\liminf_{n\to\infty}\delta(n)=\liminf_{n\to\infty}\frac{1}{n^2}\,\norm{V_r}_4^4,   \label{eqn:n2S4Vr}
\end{equation}
and then, since $\norm{V_r}_4^4\ge 0$, we have from~\eqref{eqn:liminfLB}
\[
\limsup_{n\to\infty}\max_{V\in\V_n}F(J_r+V_r)\le\limsup_{n\to\infty}F(J_r).
\]
Now use Theorem~\ref{thm:ternary} and the condition~\eqref{eqn:logn6p1-mainres} to replace $\limsup_{n\to\infty}F(J_r)$ by~$f(r)$, proving the theorem.
\par
It remains to prove the claim~\eqref{eqn:n2S4Vr}. By the condition \eqref{eqn:logn6p1-mainres}, we obtain from~\eqref{eqn:deltaS4} that
\begin{align}
\liminf_{n\to\infty}\delta(n)&=\liminf_{n\to\infty}\bigg[\frac{1}{n^2}\norm{V_r}_4^4-8p_n^{-1/2}n^{-1}(\log n)^{3/2}\norm{V_r}_4^2\bigg]   \label{eqn:deltaS4LB1}\\
&=\liminf_{n\to\infty}\bigg[\frac{1}{n^2}\norm{V_r}_4^4\bigg(1-\frac{8p_n^{-1/2}n(\log n)^{3/2}}{\norm{V_r}_4^2}\bigg)\bigg].   \label{eqn:deltaS4LB2}
\end{align}
Partition the infinite set $N$, in which $n$ takes values, into subsets $N_1$, $N_2$ defined by
\[
n\in
\begin{cases}
N_1 & \mbox{if $\norm{V_r}_4^4>   p_n^{-1}n^2(\log n)^5$}\\[1ex]
N_2 & \mbox{if $\norm{V_r}_4^4\le p_n^{-1}n^2(\log n)^5$},
\end{cases}
\]
at least one of which is infinite. If $N_1$ is infinite, then for $n\in N_1$ we have
\[
\frac{8p_n^{-1/2} n(\log n)^{3/2}}{\norm{V_r}_4^2}<\frac{8}{\log n}\to0,
\]
so that by \eqref{eqn:deltaS4LB2}, the claim~\eqref{eqn:n2S4Vr} holds when $n$ takes values only in $N_1$. On the other hand, if $N_2$ is infinite, then for $n\in N_2$ we have
\[
8p_n^{-1/2}n^{-1}(\log n)^{3/2}\norm{V_r}_4^2\le 8p_n^{-1}(\log n)^4,
\]
so that by using the condition~\eqref{eqn:logn6p1-mainres} and substituting in~\eqref{eqn:deltaS4LB1} we conclude that~\eqref{eqn:n2S4Vr} holds when $n$ takes values only in $N_2$. Since $n\in N_1\cup N_2=N$, we therefore have established the claim~\eqref{eqn:n2S4Vr}.
\qed


\section{Proof of Theorem~\ref{thm:almost-all-mainres}}
\label{sec:almost-all}

The method of the proof is to apply Proposition~\ref{prop:L4-sum} and bound $\norm{V_r}_4$ for almost all choices $V\in\V_n$, for which we require the following large deviation result (see~\cite[Thm.~A.1.16]{Alon2008}, for example).
\begin{lemma}
\label{lem:hoeffding}
Let $X_1,X_2,\dots,X_{m}$ be mutually independent random variables satisfying $\E(X_j)=0$ and $\abs{X_j}\le 1$ for $1\le j\le m$. Then, for real $a\ge 0$,
\[
\Pr\bigg(\biggabs{\sum_{j=1}^m X_j\,}^2\ge a\bigg)\le 2\,e^{-\frac{a}{2m}}.
\]
\end{lemma}
\par
We next use Lemma~\ref{lem:hoeffding} to give an upper bound for $\norm{V_r}_4$ for almost all $V\in\V_n$.
\begin{lemma}
\label{lem:Pr-V}
Let $V$ be drawn uniformly from $\V_n$ and let $r$ be real. Then, as $n\to\infty$,
\[
\Pr\Big(\norm{V_r}_4^4<288[\psi(n)]^2\log n\Big)\to 1.
\]
\end{lemma}
\begin{proof}
Given a polynomial $A\in\C[z]$ of degree at most $n-1$, it is a simple consequence of Bernstein's inequality that
\[
\max_{\abs{z}=1}\;\abs{A(z)}\le 6\max_{0\le j<4n}\;\abs{A(\z_{4n}^j)}
\]
(see~\cite[p.~691]{Spencer1985}). Therefore, by~\eqref{eqn:L4L2max},
\[
\norm{V_r}_4^4\le 36\psi(n)\max_{0\le j<4n}\;\abs{V_r(\z_{4n}^j)}^2.
\]
Hence, it is sufficient to show that
\begin{equation}
\Pr\Big(\max_{0\le j<4n}\abs{V_r(\z_{4n}^j)}^2<8\psi(n)\log n\Big)\to 1.   \label{eqn:max_V}
\end{equation}
Write $a(n)=8\psi(n)\log n$. A crude estimate gives
\begin{align}
\lefteqn{\Pr\Big(\max_{0\le j<4n}\abs{V_r(\z_{4n}^j)}^2 \ge a(n)\Big)\le \sum_{j=0}^{4n-1}\Pr\Big(\abs{V_r(\z_{4n}^j)}^2 \ge a(n)\Big)}   \nonumber\\
&\qquad\qquad\le\sum_{j=0}^{4n-1}\bigg[\Pr\Big(\bigabs{\Re\big(V_r(\z_{4n}^j)\big)}^2\ge\tfrac{1}{2}a(n)\Big)+\Pr\Big(\bigabs{\Im\big(V_r(\z_{4n}^j)\big)}^2\ge\tfrac{1}{2}a(n)\Big)\bigg].   \label{eqn:PrVr}
\end{align}
Write $V\in\V_n$ as $V(z)=\sum_{k=0}^{n-1}v_kz^k$ and note that $v_k=0$ if and only if $\gcd(k,n)=1$. Then we have by the definition of the rotation $V_r$,
\[
V_r(z)=\sum_{\stack{\ell=0}{\gcd(\ell,n)>1}}^{n-1}v_{\ell}\,z^{k(\ell)},
\]
where $k(\ell)=(\ell-\lfloor nr\rfloor)\bmod n$. Let $\lambda\in\C$ be such that $\abs{\lambda}\le 1$. Then
\begin{align*}
\Pr\Big(\bigabs{\Re\big(V_r(\lambda)\big)}^2 \ge \tfrac{1}{2}a(n)\Big)&=\Pr\Bigg(\Biggabs{\sum_{\stack{\ell=0}{\gcd(\ell,n)>1}}^{n-1}v_\ell\,\Re(\lambda^{k(\ell)}\big)}^2 \ge \tfrac{1}{2}a(n)\Bigg)\\
&\le 2 e^{-\tfrac{1}{2\psi(n)}\cdot\tfrac{a(n)}{2}}
\end{align*}
by application of Lemma~\ref{lem:hoeffding}. By definition of $a(n)$ we then obtain
\begin{align*}
\Pr\Big(\bigabs{\Re\big(V_r(\lambda)\big)}^2\ge\tfrac{1}{2}a(n) \Big)&\le 2n^{-2},
\intertext{and by similar reasoning}
\Pr\Big(\bigabs{\Im\big(V_r(\lambda)\big)}^2 \ge \tfrac{1}{2}a(n)\Big)&\le 2n^{-2}.
\end{align*}
Substitution in \eqref{eqn:PrVr} then gives
\[
\Pr\Big(\max_{0\le j<4n}\abs{V_r(\z_{4n}^j)}^2\ge a(n)\Big)\le 16n^{-1},
\]
which implies~\eqref{eqn:max_V}, as required.
\end{proof}
\par
We now use Lemma~\ref{lem:Pr-V} to prove Theorem~\ref{thm:almost-all-mainres}.
\begin{proof}[Proof of Theorem~\ref{thm:almost-all-mainres}]
Define a subset $\U_n$ of $\V_n$ by
\begin{equation}
\U_n:=\left\{V\in\V_n:\norm{V_r}_4^4<288p_n^{-2}n^2(\log n)^3\right\}.   \label{eqn:def-U}
\end{equation}
Using the upper bound~\eqref{eqn:psi-dmultiple} for $\psi(n)$, Lemma~\ref{lem:Pr-V} implies that
\begin{equation}
\frac{\abs{\U_n}}{\abs{\V_n}}\to 1.   \label{eqn:UV}
\end{equation}
By the triangle inequality,
\begin{equation}
\Biggabs{\frac{1}{F(J_r+V_r)}-\frac{1}{f(r)}}\le\Biggabs{\frac{1}{F(J_r+V_r)}-\left(\frac{\phi(n)}{n}\right)^2\frac{1}{F(J_r)}}+\Biggabs{\left(\frac{\phi(n)}{n}\right)^2\frac{1}{F(J_r)}-\frac{1}{f(r)}}.   \label{eqn:triangle}
\end{equation}
Using the condition~\eqref{eqn:logn6p1-2-mainres} and the growth rate~\eqref{eqn:growth-phi} of $\phi(n)$, we find from Theorem~\ref{thm:ternary} that
\begin{equation}
\Biggabs{\left(\frac{\phi(n)}{n}\right)^2\frac{1}{F(J_r)}-\frac{1}{f(r)}}\to 0.   \label{eqn:triangle_term1}
\end{equation}
From Proposition~\ref{prop:L4-sum} we have
\begin{equation}
\Biggabs{\frac{1}{F(J_r+V_r)}-\left(\frac{\phi(n)}{n}\right)^2\!\!\frac{1}{F(J_r)}}<\gamma(n)\quad\mbox{for $V\in\U_n$},   \label{eqn:triangle_term2}
\end{equation}
where
\begin{align*}
\gamma(n)&=\max_{V\in\U_n}\bigg(\frac{1}{n^2}\,\norm{V_r}_4^4+8p_n^{-1/2}n^{-1}(\log n)^{3/2}\norm{V_r}_4^2+58p_n^{-1/2}(\log n)^{7/2}\bigg)   \nonumber\\
&<8p_n^{-2}(\log n)^3+\sqrt{512}\,p_n^{-3/2}(\log n)^3+58p_n^{-1/2}(\log n)^{7/2},
\end{align*}
by the definition~\eqref{eqn:def-U} of $\U_n$. Using the condition~\eqref{eqn:logn6p1-2-mainres}, we have $\gamma(n)\to0$. Since $\U_n$ forms a set of measure~$1$ within $\V_n$ by~\eqref{eqn:UV}, we find by substitution of~\eqref{eqn:triangle_term1} and~\eqref{eqn:triangle_term2} into~\eqref{eqn:triangle} that
\[
\Biggabs{\frac{1}{F(J_r+V_r)}-\frac{1}{f(r)}}\to 0\quad\mbox{in probability}.
\]
Since $f(r)$ takes values only in a finite interval bounded away from~$0$, we then have
\[
\abs{F(J_r+V_r)-f(r)}\to 0\quad\mbox{in probability},
\]
which completes the proof.
\end{proof}


\section{Proof of Theorem~\ref{thm:product-mainres}}
\label{sec:typical}

From Proposition~\ref{prop:L4-sum} we have
\begin{equation}
\Biggabs{\frac{1}{F(J_r+V_r)}-\left(\frac{\phi(n)}{n}\right)^2\!\!\frac{1}{F(J_r)}}<\gamma(n),   \label{eqn:atmost-gamma-2}
\end{equation}
where 
\begin{equation}
\gamma(n)=\frac{1}{n^2}\,\norm{V_r}_4^4
    + 8p_n^{-1/2}n^{-1}(\log n)^{3/2}\norm{V_r}_4^2 
    + 58p_n^{-1/2}(\log n)^{7/2}.   \label{eqn:gamma-2}
\end{equation}
We also have from~\eqref{eqn:L4L2max}, Lemma~\ref{lem:poly-bound},~\eqref{eqn:QA-shift}, and the upper bound~\eqref{eqn:psi-dmultiple} for $\psi(n)$,
\begin{equation}
\norm{V_r}_4^4\le(2\log n)^2\Big(\max_{0\le k<n}\bigabs{V(\z_n^k)}^2\Big)p_n^{-1}n\log n.   \label{eqn:S4Vr2logn}
\end{equation}
We now bound the term $\abs{V(\z_n^k)}$. By definition of $V$, we have for integer $k$,
\begin{align*}
V(\z_n^k) 
&= \sum_{\stack{j=0}{\gcd(j,n)>1}}^{n-1} 
     \bigleg{j}{\tfrac{n}{\gcd(j,n)}}\, \z_n^{kj} \\
&= \sum_{\stack{0<m<n}{m \mid n}}
     \sum_{\stack{\ell=0}{\gcd(\ell,m)=1}}^{m-1}
     \bigleg{\tfrac{\ell n}{m}}{m}\, \z_m^{k\ell} 
\intertext{by putting $m = n/\gcd(j,n)$, so that we must have
$j=\ell n/m$ where, since $n$ is square-free, $0\le\ell<m$ and $\gcd(\ell,m)=1$. Since the Jacobi symbol is multiplicative, and $\leg{\ell}{m}=0$ for $\gcd(\ell,m)>1$, we then have}
V(\z_n^k) 
&=\sum_{\stack{0<m<n}{m \mid n}} \bigleg{\tfrac{n}{m}}{m}
    \sum_{\ell=0}^{m-1} \leg{\ell}{m}\, \z_m^{k\ell},
\intertext{and therefore}
\abs{V(\z_n^k)}
&\le \sum_{\stack{0<m<n}{m\mid n}}\biggabs{\sum_{\ell=0}^{m-1} \leg{\ell}{m}\,\z_m^{k\ell}}\\
&\le \sum_{\stack{0<m<n}{m \mid n}} m^{1/2}
\intertext{by Lemma~\ref{lem:gauss-sum}. Hence,}
\bigabs{V(\z_n^k)}
&\le \sum_{j=1}^{\omega(n)}\binom{\omega(n)}{j}\left(\frac{n}{p_n^j}\right)^{1/2}\\
&< n^{1/2}\big(1+p_n^{-1/2}\big)^{\omega(n)} \\
&\le n^{1/2}\big(1+(\log n)^{-7/2}\big)^{\log n} 
\end{align*}
for all sufficiently large~$n$, by \eqref{eqn:logn6p1-3-mainres}
and~\eqref{eqn:wnlogn}. Therefore
\[
\abs{V(\z_n^k)} = O(n^{1/2}).
\]
Substitute in~\eqref{eqn:S4Vr2logn} to give
\[
\norm{V_r}_4^4=O\big(p_n^{-1}n^2(\log n)^3\big),
\]
and then substitute in~\eqref{eqn:gamma-2} to show that
\[
\gamma(n)=O\big(p_n^{-1}(\log n)^3\big)+O\big(p_n^{-1}(\log n)^3\big)+O\big(p_n^{-1/2}(\log n)^{7/2}\big)\to0,
\]
by the condition~\eqref{eqn:logn6p1-3-mainres}. The required result then follows from~\eqref{eqn:atmost-gamma-2} and Theorem~\ref{thm:ternary}, using the growth rate~\eqref{eqn:growth-phi} of~$\phi(n)$ and the condition~\eqref{eqn:logn6p1-3-mainres}.
\qed


\section{Proof of Theorem~\ref{thm:twoprimes-mainres}}
\label{sec:two-primes}

Let $V\in\V_n$. From Proposition~\ref{prop:L4-sum} we have
\begin{equation}
\Biggabs{\frac{1}{F(J_r+V_r)}-\left(\frac{\phi(n)}{n}\right)^2\!\!\frac{1}{F(J_r)}}<\gamma(n),   \label{eqn:atmost-gamma-3}
\end{equation}
where
\[
\gamma(n)
=\frac{1}{n^2}\,\norm{V_r}_4^4
    + 8p_n^{-1/2}n^{-1}(\log n)^{3/2}\norm{V_r}_4^2 
    + 58p_n^{-1/2}(\log n)^{7/2}.
\]
From the upper bound~\eqref{eqn:norm_V_r_bound} for~$\norm{V_r}_4^4$ and the upper bound~\eqref{eqn:psi-omega} for $\psi(n)$ we have $\norm{V_r}_4^4\le (2n/p_n)^3$ for all sufficiently large $n$ since the condition~\eqref{eqn:cond-2primes-mainres} forces $\w(n)\le 2$ for all sufficiently large~$n$. Hence
\[
\gamma(n)
=O\big(p_n^{-3}n\big) 
    + O\big(p_n^{-2}n^{1/2}(\log n)^{3/2}\big)
    + O\big(p_n^{-1/2}(\log n)^{7/2}\big).
\]
By the condition~\eqref{eqn:cond-2primes-mainres} we then have $\gamma(n)\to0$, and the required result follows from \eqref{eqn:atmost-gamma-3} and Theorem~\ref{thm:ternary}, using the growth rate~\eqref{eqn:growth-phi} of~$\phi(n)$ and the condition~\eqref{eqn:cond-2primes-mainres}.
\qed

\providecommand{\bysame}{\leavevmode\hbox to3em{\hrulefill}\thinspace}
\providecommand{\MR}{\relax\ifhmode\unskip\space\fi MR }
\providecommand{\MRhref}[2]{%
  \href{http://www.ams.org/mathscinet-getitem?mr=#1}{#2}
}
\providecommand{\href}[2]{#2}


\begin{thebibliography}{10}

\bibitem{Alon2008}
N.~Alon and J.~H. Spencer, \emph{The probabilistic method}, 3rd ed., Wiley,
  Hoboken, New Jersey, 2008.

\bibitem{Apostol1976}
T.~M. Apostol, \emph{Introduction to analytic number theory}, Springer, New
  York, 1976.

\bibitem{Beck1991}
J.~Beck, \emph{Flat polynomials on the unit circle---note on a problem of
  {L}ittlewood}, Bull. London Math. Soc. \textbf{23} (1991), no.~3, 269--277.

\bibitem{Beenker1985}
G.~F.~M. Beenker, T.~A. C.~M. Claasen, and P.~W.~C. Hermens, \emph{Binary
  sequences with a maximally flat amplitude spectrum}, Philips J.\ Res.
  \textbf{40} (1985), 289--304.

\bibitem{Bernasconi1987}
J.~Bernasconi, \emph{Low autocorrelation binary sequences: statistical
  mechanics and configuration state analysis}, J. Physique \textbf{48} (1987),
  559--567.

\bibitem{Berndt1998}
B.~C. Berndt, R.~J. Evans, and K.~S. Williams, \emph{{G}auss and {J}acobi
  sums}, John Wiley \& Sons, New York, NY, 1998.

\bibitem{Borwein2002a}
P.~Borwein, \emph{Computational excursions in analysis and number theory}, CMS
  Books in Mathematics, Springer-Verlag, New York, NY, 2002.

\bibitem{Borwein2001}
P.~Borwein and K.-K.~S. Choi, \emph{Merit factors of polynomials formed by
  {J}acobi symbols}, Canad. J. Math. \textbf{53} (2001), no.~1, 33--50.

\bibitem{Borwein2002}
\bysame, \emph{Explicit merit factor formulae for {F}ekete and {T}uryn
  polynomials}, Trans.\ Amer.\ Math.\ Soc. \textbf{354} (2002), no.~1,
  219--234.

\bibitem{Borwein2004}
P.~Borwein, K.-K.~S. Choi, and J.~Jedwab, \emph{Binary sequences with merit
  factor greater than $6.34$}, IEEE Trans. Inf. Theory \textbf{50} (2004),
  no.~12, 3234--3249.

\bibitem{Borwein2001b}
P.~Borwein, K.-K.~S. Choi, and S.~Yazdani, \emph{An extremal property of
  {F}ekete polynomials}, Proc. Amer. Math. Soc. \textbf{129} (2001), no.~1,
  19--27.

\bibitem{Borwein2001a}
P.~Borwein and R.~Lockhart, \emph{The expected {$L_p$} norm of random
  polynomials}, Proc. Amer. Math. Soc. \textbf{129} (2001), no.~5, 1463--1472.

\bibitem{Conrey2000}
B.~Conrey, A.~Granville, B.~Poonen, and K.~Soundararajan, \emph{Zeros of
  {F}ekete polynomials}, Ann. Inst. Fourier (Grenoble) \textbf{50} (2000),
  no.~3, 865--889.

\bibitem{Erdos1962}
P.~Erd{\H{o}}s, \emph{An inequality for the maximum of trigonometric
  polynomials}, Ann. Polon. Math. \textbf{12} (1962), 151--154.

\bibitem{Golay1982}
M.~J.~E. Golay, \emph{The merit factor of long low autocorrelation binary
  sequences}, IEEE Trans. Inf. Theory \textbf{IT-28} (1982), no.~3, 543--549.

\bibitem{Golay1983}
\bysame, \emph{The merit factor of {Legendre} sequences}, IEEE Trans. Inf.
  Theory \textbf{29} (1983), no.~6, 934--936.

\bibitem{Hardy1979}
G.~H. Hardy and E.~M. Wright, \emph{An introduction to the theory of numbers},
  5th ed., Oxford Science Publications, Oxford, 1979.

\bibitem{Hoholdt1988}
T.~H{\o}holdt and H.~E. Jensen, \emph{Determination of the merit factor of
  {L}egendre sequences}, IEEE Trans. Inf. Theory \textbf{34} (1988), no.~1,
  161--164.

\bibitem{Jedwab2005}
J.~Jedwab, \emph{A survey of the merit factor problem for binary sequences},
  Proc. of Sequences and Their Applications (SETA), Lecture Notes in Computer
  Science, vol. 3486, New York: Springer Verlag, 2005, pp.~30--55.

\bibitem{Jensen1991}
J.~M. Jensen, H.~E. Jensen, and T.~H{\o}holdt, \emph{The merit factor of binary
  sequences related to difference sets}, IEEE Trans. Inf. Theory \textbf{37}
  (1991), no.~3, 617--626.

\bibitem{Littlewood1966}
J.~E. Littlewood, \emph{On polynomials {$\sum^n\pm z^m$, $\sum^ne^{\alpha_mi}
  z^m$, $z=e^{\theta i}$}}, J. London Math. Soc. \textbf{41} (1966), 367--376.

\bibitem{Littlewood1968}
\bysame, \emph{Some problems in real and complex analysis}, Heath Mathematical
  Monographs, D. C. Heath and Company, Lexington, MA, 1968.

\bibitem{Montgomery1980}
H.~L. Montgomery, \emph{An exponential polynomial formed with the {L}egendre
  symbol}, Acta Arith. \textbf{37} (1980), 375--380.

\bibitem{Newman1960}
D.~J. Newman, \emph{Norms of polynomials}, Amer. Math. Monthly \textbf{67}
  (1960), 778--779.

\bibitem{Newman1990}
D.~J. Newman and J.~S. Byrnes, \emph{The {$L^4$} norm of a polynomial with
  coeffiecients {$\pm 1$}}, Amer. Math. Monthly \textbf{97} (1990), no.~1,
  42--45.

\bibitem{Paterson2000}
K.~G. Paterson and V.~Tarokh, \emph{On the existence and construction of good
  codes with low peak-to-average power ratios}, IEEE Trans. Inf. Theory
  \textbf{46} (2000), no.~6, 1974--1987.

\bibitem{Schmidt2009}
K.-U. Schmidt, J.~Jedwab, and M.~G. Parker, \emph{Two binary sequence families
  with large merit factor}, Adv. Math. Commun. \textbf{3} (2009), no.~2,
  135--156.

\bibitem{Spencer1985}
Joel Spencer, \emph{Six standard deviations suffice}, Trans. Amer. Math. Soc.
  \textbf{289} (1985), no.~2, 679--706.

\bibitem{Turyn1968}
R.~J. Turyn, \emph{Sequences with small correlation}, Error Correcting Codes
  (Henry~B. Mann, ed.), Wiley, New York, 1968, pp.~195--228.

\bibitem{Xiong2008}
T.~Xiong and J.~I. Hall, \emph{Construction of even length binary sequences
  with asymptotic merit factor 6}, IEEE Trans. Inf. Theory \textbf{54} (2008),
  no.~2, 931--935.

\bibitem{Xiong2011a}
\bysame, \emph{Modifications on character sequences and construction of large
  even length binary sequences}, Preprint (2010).

\bibitem{Xiong2011}
\bysame, \emph{Modifications of modified {J}acobi sequences}, IEEE Trans. Inf.
  Theory \textbf{57} (2011), no.~1, 493--504.

\end{thebibliography}

\end{document}